\documentclass{amsart}

\usepackage[utf8]{inputenc}%
\usepackage[T1]{fontenc}%
\usepackage{amssymb,epic,eepic,epsfig,amsmath,mathdots}
\usepackage{mathrsfs,microtype,geometry,mathtools}
\usepackage{amsfonts}
\usepackage{amsthm}
\usepackage{color,enumitem}
\usepackage{hyperref}
\mathtoolsset{mathic}

\setlist[description]{topsep=9pt,itemsep=6pt}

\DeclareMathOperator{\col}{Col}
\DeclareMathOperator{\com}{\mathscr{C}}
\DeclareMathOperator{\Ex}{Ex}

\numberwithin{equation}{section}
\theoremstyle{plain}
\newtheorem{theorem}{Theorem}
\newtheorem{proposition}{Proposition}[section]

\newtheorem{observation}[proposition]{Observation}

\newtheorem{assert}[proposition]{Assertion}

\theoremstyle{definition}
\newtheorem{definition}[proposition]{Definition}

\newtheorem{procedure}{Procedure}

\theoremstyle{remark}

\newcommand{\abs}[1]{\left\lvert#1\right\rvert}

\newcommand{\U}{\mathbf{\theta}}
\newcommand{\s}[2]{\sharp(#1\hookrightarrow#2)}
\newcommand{\sst}[2]{\left\{#1\,:\,#2\right\}}

\newcommand{\e}{\mathrm{e}^}

\def\BN{\mathbf N}

\def\T{\mathcal T}

\def\A{\mathcal A}

\def\a{\alpha}

\renewcommand{\le}{\leqslant}
\renewcommand{\ge}{\geqslant}
\renewcommand{\leq}{\leqslant}
\renewcommand{\geq}{\geqslant}

\begin{document}

\title[Isomorphism of Weighted Trees and Stanley's Isomorphism Conjecture]{Isomorphism of Weighted Trees
and Stanley's Isomorphism Conjecture for Caterpillars}

\author{Martin Loebl}
\address{Dept.~of Applied Mathematics\\
Charles University\\
Praha\\
Czech Republic.}
\email{loebl@kam.mff.cuni.cz}

\author{Jean-S\'ebastien Sereni}
\address{Centre National de la Recherche Scientifique\\
(CSTB)\\
Strasbourg\\
France.}
\email{sereni@kam.mff.cuni.cz}

\date{\today}
\thanks{This work was done within the scope of the International Associated Laboratory STRUCO}
\thanks{The authors were partially supported by the
Czech Science Foundation under the contract number~P202-13-21988S (M. L.) and
by P.H.C. Barrande~31231PF of the French M.A.E. (J.-S. S.)}

\begin{abstract} 
This paper contributes to a programme initiated by the first author:
`How much information about a graph is revealed in its Potts partition function?'.
We show that the $W$-polynomial distinguishes non-isomorphic weighted trees of
a \emph{good} family. The framework developed to do so also allows us to show that
the $W$-polynomial distinguishes non-isomorphic caterpillars. This establishes
Stanley's isomorphism conjecture for caterpillars, an extensively studied problem.
\end{abstract}


\maketitle

\section{Introduction}\label{sec:intro} 
Consider the following data set~$D(T)$ associated with a tree~$T$: for every integer~$n$ and
every partition~$P$ of~$n$, we are given the number of subsets~$X$ of edges
of~$T$ such that $P$ is equal to the multiset formed by the orders of the
components of~$T-X$. Note that this number is~$0$ if $n$ is not the number of
vertices of~$T$. Note also that if $P$ is composed of~$t$ integers, the
corresponding subsets~$X$, if any, all have cardinality~$t-1$.  For instance,
one can determine the number of vertices of~$T$ by checking, for each positive
integer~$n$, whether the trivial partition~$\{n\}$ returns a non-zero value
(which, necessarily, will be~$1$).  Once the number~$n$ of vertices of~$T$ is
known, the number of leaves of~$T$ is precisely the number returned by the
partition $\{n-1,1\}$, which corresponds to the number of edges~$e$
such that $T-e$ has one component of order~$1$.  The problem is to know whether
this information distinguishes non-isomorphic trees.  In other words, if~$T$
and~$T'$ are two trees such that $D(T)=D(T')$, is it true that necessarily~$T$
and~$T'$ are isomorphic?  That such a reconstruction is always possible was
suggested by different authors.  We note that there could be non-constructive
proofs of the statement. Thus it is a different (harder) problem to be able to
effectively recover the tree~$T$ from the knowledge of~$D(T)$. We explain in
subsections~\ref{sub.11},~\ref{sub.22} and~\ref{sub.33} why studying the strength of
the information contained in~$D(T)$ for an arbitrary tree~$T$ helps to
understand the strength of the partition function of the Potts model in
a magnetic field, for general graphs.

\subsection{State of the Art.}\label{sub.start}  
Extensive efforts were dedicated (personal communication with Noble) to proving
that $D(T)$ distinguishes non-isomorphic caterpillars ---  a \emph{caterpillar}
is a tree where all edges not incident with a leaf form a path, and a
\emph{leaf} is a vertex of degree one. Part of the Ph.D. thesis of
Zamora~\cite{Zam13} (under the supervision of M.~L.) is dedicated to this
problem. In addition, Aliste-Prieto and Zamora~\cite{AlZa14}, established the
statement restricted to the class of proper caterpillars: a caterpillar is
\emph{proper} if every vertex is a leaf or adjacent to a leaf.  Prior to that,
partial results had been obtained by Martin, Morin and~Wagner~\cite{MMW08} who
had established the statement for a subclass of proper caterpillars (where no
two non-leaf vertices are adjacent to the same number of leaves) and also to
the class of spiders, which is composed of all trees with a unique vertex of
degree greater than~two.  Other related results were obtained by Orellana
and Scott~\cite{OS}, Smith, Smith and Tian~\cite{SST} or can be found in the
undergraduate thesis by~Fougere~\cite{Fou03} and the MSc~thesis
by~Morin~\cite{Mor05}.  Finally, Hell and Ji~\cite{HJ18} have verified by computer that
Stanley's isomorphism conjecture~\cite{Sta95}, which we present in
Subsection~\ref{sub.22}, is true for trees with at most~$29$ vertices.
Previously, Russel has verified by computer that Stanley's
isomorphism conjecture is true for trees with at most~$25$ vertices (the code
is available at \url{https://github.com/keeler/csf}) and it was reported
(see~\cite[p.~238]{MMW08}) that Tan verified it for trees with at most~$23$
vertices.


\subsection{Main Contribution.}\label{sub.main} 
We solve affirmatively Stanley's isomorphism conjecture restricted to the class of caterpillars. 
We also investigate a weighted version of the problem, bearing in
mind its connections with graph polynomials, graph colouring and the Potts
model.  First we summarise the background and motivations. 

\section{Motivation}\label{s.mot} 
In this section we summarise the background (the Noble and Welsh conjecture and the Stanley conjecture) and describe our motivation.

\subsection{The Noble and Welsh Conjecture.}\label{sub.11} 
Motivated by the combinatorial aspects of the relationship between chord
diagrams and Vassiliev invariants of knots, Noble and Welsh~\cite{NoWe99}
introduced a polynomial of weighted graphs, the $W$-polynomial, which includes
several specialisations in combinatorics, such as the Tutte polynomial, the
matching polynomial (of ordinary graphs) and the polymatroid polynomial of
Oxley and Whittle~\cite{OxWh93}. We need to introduce some terminology to
define $W$.

A \emph{weighted graph} is a graph~$G=(V,E)$ together with a function~$w\colon
V\to\mathbf{Z}^+$.  The \emph{weight} of a subset $V'$ of vertices is
$w(V')\coloneqq\sum_{v\in V'}w(v)$. If $A\subseteq E$, we let~$c_V(A)$ be the
number of components of the graph $(V,A)$, where we may omit the subscript when
there is no risk of confusion.  Further, let~$n_1,\dotsc,n_{c(A)}$ be the
weights of the vertex sets of these components, listed in decreasing order:
$n_1\ge\dotsb\ge n_{c(A)}$.  We write~$x(A)$ to mean~$\prod_{i=1}^{c(A)}x_{n_i}$.  Let
\[
    W_G(z,x_1,x_2,\dotsc)\coloneqq\sum_{A\subseteq E}x(A){(z-1)}^{\abs{A}-\abs{V}+c(A)}.
\]
In particular, $W_G$ depends on~$z$ if and only if $G$ contains
a cycle~\cite[Proposition 5.1]{NoWe99}.  Unlike the Tutte polynomial, the
$W$-polynomial is $\# P$-hard to compute even for trees~\cite[Theorems~7.3
and~7.12]{NoWe99} and for complete graphs~\cite[Theorems~7.11
and~7.14]{NoWe99}.

In the case of unweighted graphs, which corresponds here to the weight
function~$w$ being identically~$1$, Noble and Welsh refer to the $W$-polynomial
as the \emph{$U$-polynomial}.  While computing~$W$ is hard for complete graphs,
Annan~\cite{Ann94} proved that $U_{K_n}(z,x_1,x_2,\dotsc)$ can be computed in
polynomial time, which is also the case for the Tutte polynomial.
However,~$U$ also exhibits differences with the Tutte polynomial: while finding
two non-isomorphic graphs with the same Tutte polynomial is easy, the same problem
is harder for~$U$. Brylawski~\cite{Bry81} found two non-isomorphic graphs with the same polychromate, 
and Sarmiento~\cite{Sar99} proved that the $U-$polynomial is equivalent to Brylawski's polychromate.
But the question remains open for trees: does the $U$-polynomial distinguishes
non-isomorphic trees? That this is the case became known as the \emph{Noble and
Welsh conjecture}.  This is clearly equivalent to our initial question:
`Does~$D(T)$ distinguish non-isomorphic trees?'

Noble and Welsh demonstrated the $U$-polynomial to be equivalent to the \emph{symmetric
function generalisation of the chromatic polynomial}, a function introduced
by Stanley~\cite{Sta95}.


\subsection{Stanley's isomorphism Conjecture.}\label{sub.22} 
To introduce Stanley's isomorphism conjecture let us first define graph
colouring.  A \emph{colouring} of a graph~$G=(V,E)$ is a mapping~$s\colon
V\to\BN^+$. We define~$b(s)$ to be the number of \emph{monochromatic edges}
in~$s$, that is, the number of edges~$uv$ such that $s(u)=s(v)$.  The mapping~$s$
is a \emph{$k$-colouring} if $s(V)\subseteq\{1,\dotsc,k\}$ and $s$ is
\emph{proper} if $b(s)=0$, that is, $s(u)\neq s(v)$ whenever $u$ and $v$ are
two adjacent vertices of~$G$.  We let $\col(G;k)$ be the set of proper
$k$-colourings of~$G$ and~$\col(G)$ be the set of all proper colourings of~$G$.

In the mid 1990s, Stanley~\cite{Sta95} introduced the \emph{symmetric function
generalization of the chromatic polynomial}, defined to be
\[
      X_G(x_1,x_2,\dotsc)\coloneqq\sum_{s\in\col(G)}\prod_{v\in V}x_{s(v)}.
\]
This is a homogeneous symmetric function in~$(x_1,x_2,\dotsc)$ of degree
$\abs{V}$.  As one might expect, $X_G$ does not distinguish non-isomorphic graphs:
there exist two non-isomorphic graphs on $5$ vertices with the same function
$X$. However, Stanley~\cite{Sta95} asked whether the polynomial $X_G$
distinguishes non-isomorphic trees.  The assertion that it does became known as
\emph{Stanley's isomorphism conjecture}. 

Further, Stanley~\cite{Sta98} later initiated the study of a common
generalisation of $X$ and the Tutte polynomial, namely the \emph{symmetric
function generalisation of the bad colouring polynomial}, defined for every
graph $G=(V,E)$ by
\[
    X_G(t,x_1,x_2,\dotsc)\coloneqq\sum_{s\colon V\to\BN^+}{(1+t)}^{b(s)}\prod_{v\in V}x_{s(v)}.
\]
Note that the sum runs over all colourings of~$G$, not only the proper ones.
Noble and Welsh~\cite[Theorem~6.2]{NoWe99} proved $X_G(t,x_1,x_2,\dotsc)$ to be
equivalent to the $U$-polynomial of~$G$.

\subsection{Loebl's Conjectures.}\label{sub.33} 
Loebl~\cite{Loe07} introduced the $q$-chromatic functions.
Let~$k\in\mathbf{N}$.   The \emph{$q$-chromatic function} of a
graph~$G=(V,E)$ is
\begin{equation}\label{e.poly1}
M_G(k,q)\coloneqq\sum_{s\in \col(G;k)} q^{\sum_{v\in V} s(v)}.
\end{equation}
It is known~\cite{Loe07} that
\[
    M_G(k,q)=\sum_{A\subset E} {(-1)}^{\abs{A}}\prod_{C\in \com(A)} {(k)}_{q^{\abs{C}}},
\]
where~the \emph{quantum integer}~${(k)}_r$ is~$r^{k-1}+\dotsb + r+ 1$
and~$\com(A)$ is the set of components of the spanning subgraph~$(V,A)$
while~$\abs{C}$ is the number of vertices in the component~$C$. Moreover Loebl
also introduced the \emph{$q$-dichromate}, defined as
\[
    B_G(x,y,q)\coloneqq\sum_{A\subset E} x^{\abs{A}} \prod_{C\in\com(A)} {(y)}_{q^{\abs{C}}}.
\]
Loebl~\cite{Loe07} conjectured the following.
\begin{itemize}
\item The $q$-dichromate is equivalent to the $U$-polynomial.
\item The $U$-polynomial distinguishes non-isomorphic chordal graphs.
\end{itemize}

There could be a close link between the latter conjecture and that of Stanley:
chordal graphs have a very distinguished tree structure.
Indeed, a folklore theorem~\cite{BLS99} states that
the class of chordal graphs is precisely the class of intersection graphs of subtrees of a tree,
that is, for each chordal graph~$G$, there exists a tree~$T$
and a mapping $f$ that assigns to each vertex of~$G$ a subtree~$T$ such that:
two vertices~$u$ and~$v$ of~$G$ are adjacent if and only if $f(u)\cap f(v)\neq\varnothing$.

The motivation for Loebl's conjectures is formula~\eqref{bpotts} below, which connects 
the $k$-state Potts model partition function and the $q$-dichromate.

\medskip

\textbf{The Potts model.}
We consider a standard model where magnetic materials are represented as lattices: vertices
are atoms and weighted edges are nearest-neighbourhood interactions. We assume that each
atom has one out of $k$ possible magnetic moments, for a fixed positive integer~$k$.
Thus we set~$S\coloneqq\{0,\dotsc,k-1\}$.
Every element of~$S$ is called a \emph{spin}.
A \emph{state} of a graph~$G=(V,E)$ is then an assignment of a single spin to each vertex
of~$G$, that is, a function~$s\colon V\rightarrow S$. 
We assume that all the coupling constants
(nearest-neighbourhood interactions) are equal to a constant~$J$.
For each state~$s$, the \emph{Potts model energy of the state~$s$} is defined
to be~$E(P^k)(s)\coloneqq\sum_{uv\in E}J\delta(s(u), s(v))$ where,
as is customary,~$\delta$ is the Kronecker delta function defined
by~$\delta(a,b)\coloneqq1$ if $a=b$ and $\delta(a,b)\coloneqq0$ otherwise.
The \emph{$k$-state Potts model partition function} is
\[
      \sum_{s:V\rightarrow S}M(s, J)e^{E(P^k)(s)}
\]
where $M(s,J)$ is a function describing the magnetic field contribution.

Loebl proved that for each real~$J$,
\begin{equation}\label{bpotts}
B_G(e^{J}-1,k,q)=\sum_{s:V\rightarrow S}q^{\sum_{v\in V}s(v)}e^{E(P^k)(s)}.
\end{equation}
This means that the $q$-dichromate  specializes to the $k$-state Potts model partition function
with a certain magnetic field contribution.

Recently a variant of the $q$-dichromate, $B_{r,G}(x,k,q)$, was
proposed by Klazar, Loebl and Moffatt~\cite{KLM14}:
\[
B_{r,G}(x,k,q)\coloneqq\sum_{A\subseteq E}  x^{\abs{A}}  \prod_{C\in
\com(A)}   \sum_{i=0}^{k-1}   r^{\abs{C} q^{i}}.
\]
They established that if $(k,r)\in\BN^2$ with $r>1$ and $x\coloneqq\e{\beta J}-1$,
then
\begin{equation}\label{e.p5a}
B_{r,G}(x,k,q) =  \sum_{\sigma \colon V\rightarrow S}
\e{\beta  \sum_{ uv \in E(G) }   J \delta(\sigma(u), \sigma(v)) }
r^{\sum_{v\in V} q^{\sigma(v)}}.
\end{equation}
Hence $B_{r,G}(x,k,q)$ is the $k$-state Potts model partition function with
magnetic field contribution $r^{\sum_{v\in V} q^{\sigma(v)}}$.
They also proved~$B_{r,G}$ to be equivalent to~$U_G$, which can be seen as a first
step towards Loebl's programme:
\medskip

\emph{The polynomial~$U_G$ is equivalent to the Potts partition function of~$G$ with a magnetic field contribution.}

\medskip

A well-known fact is that the
isomorphism problem for general graphs is equivalent to the isomorphism
problem restricted to chordal graphs: given a graph~$G= (V,E)$, consider the
chordal graph~$G'= (V', E')$ so that $V'\coloneqq V\cup E$ and $E'=
\binom{V}{2} \cup \sst{\{u,e\},\{v,e\}}{\{u,v\}=e \in E}$.  It clearly holds
that $G$ and~$H$ are isomorphic if and only if $G'$ and~$H'$ are isomorphic.
It thus seems particularly interesting to determine whether the $U$-polynomial
does distinguish non-isomorphic chordal graphs, as conjectured by~Loebl. If true,
we would obtain a surprising conclusion:
\medskip

\emph{The Potts partition function with a magnetic field contribution contains
essentially (modulo a simple preprocessing) all the information about the
underlying graph.}

\medskip

In that respect, it seems natural to study weighted trees. The tree mentioned in
the characterisation of the class of chordal graphs can be chosen to be
a \emph{clique-tree}, where the vertices of the tree are the maximal cliques of
the graph. Now, if $v$ is a vertex of a weighted tree with weight~$w(v)$, one can
think of~$v$ as a clique of order~$w(v)$, thus obtaining an unweighted chordal
graph.  This is what motivates working in the (seemingly harder) setting of
weighted trees.

\subsection{Main Results.}\label{sub.mres} 
Two weighted graphs are \emph{isomorphic} if there is an isomorphism of the graphs
that preserves the vertex weights.
We also consider weighted trees rooted at a vertex: an isomorphism between rooted weighted trees preserves the weights by definition, but may not preserve the roots. If it does preserve the roots, then it is an \emph{r-isomorphism}.
(In particular, two rooted weighted trees that are r-ismomorphic are isomorphic but the converse is not necessarily true.)

The first purpose of this work is to prove that
the $W$-polynomial distinguishes non-isomorphic \emph{weighted trees} when
restricting to collections of weighted trees satisfying some properties made
precise later. We call any such collection a \emph{good family}. We consider this result as
a first observation towards understanding Stanley's isomorphism conjecture for the class of chordal graphs;
even though we do not know natural examples of good families of weighted trees which were studied before.
We remark that the $W$-polynomial does not distinguish general weighted trees; a simple example consists of two paths with weight sequences $1,2,1,3,2$ and $1,3,2,1,2$.

Let~$(T,w)$ be a weighted tree. We write~$V(T)$ and~$E(T)$ for the vertex set
and the edge set of~$T$, respectively. We define~$\Ex(T)$ to be the multi-set
composed of all the vertex weights (with multiplicities) of~$T$. If $e\in
E(T)$, then $T-e$ is the disjoint union of two trees, which we consider to be
weighted and rooted at the endvertex of~$e$ that they contain. A rooted
weighted tree~$(S,w_S)$ is a \emph{shape} of~$(T,w)$ if
$2\le\abs{V(S)}\le\abs{V(T)}-2$ and there exists an edge~$e\in E(T)$ such that
$S$ is one of the two components of~$T-e$; moreover $w_S$ is the restriction
of~$w$ to the vertex set of~$S$.  We consider~$S$ to be rooted at the end-vertex
of~$e$.  We usually shorten the notation and write~$S$ for the shape~$(S,w_S)$.
In a tree, a vertex of degree one is called a \emph{leaf}.
\begin{definition}\label{def.good}
A set~$\T$ of weighted trees~$(T,w)$ is \emph{good} if it satisfies the following properties.
\begin{enumerate}[label=(\arabic*)]
  \item\label{good-struct} If a vertex of~$T$ is adjacent to a leaf, then all
     its neighbours but possibly one are leaves.
  \item\label{good-weileaf} If $v$ is a leaf or has a neighbour that is a
     leaf, then $w(v)=1$.
  \item\label{good-weishape} Let~$(T,w), (T',w')\in \T$ and let~$S$ be a shape
        of~$T$ and such that $w(S)\leq w(T)/2$.  Let~$S'$ be a shape of~$T'$
        such that $\Ex(S')=\Ex(S)$.  Then $S'$ and~$S$ are r-isomorphic.
\end{enumerate}
\end{definition}
\begin{theorem}\label{thm.mmm}
The $W$-polynomial distinguishes non-isomorphic weighted trees in any good set.
\end{theorem}

Our proof of Theorem~\ref{thm.mmm} is not constructive in the sense that we are
not able to reconstruct the weighted tree~$(T, w)$ from~$W_{(T,w)}$. The
difficulty in proving the theorem is that while the main defining property of
a good family is about shapes, the $W$-polynomial does not ``see'' shapes.

\medskip\noindent

However, shapes turn out to be a useful and rather powerful notion: it
allowed us to unlock the case of general caterpillars, thereby confirming
Stanley's isomorphism conjecture for the class of (general) caterpillars.
\begin{theorem}\label{thm.mmm1}
Each caterpillar can be reconstructed from its $U$-polynomial.
\end{theorem}

\noindent
Note that Theorem~\ref{thm.mmm1}, contrary to Theorem~\ref{thm.mmm}, allows for
a full reconstruction of the tree.


\section{The Structure of the Proofs} 

We write down a procedure and with its help prove both theorems. The rest of
the paper then describes our realisation of the procedure.  We fix a good set
of weighted trees and, from now on, we say that a weighted tree is \emph{good}
if it belongs to this set. 

A \emph{$j$-form} is an r-isomorphism class of rooted weighted trees with total
weight $j$. Thus a $j$-form~$F$ is a collection of r-isomorphic rooted weighted trees and,
viewing a shape of a tree~$T$ as a rooted weighted tree, a shape can belong to
a~$j$-form. Note in particular that two shapes~$S$ and~$S'$ of a weighted tree belong to the same $j$-form for some $j$
if and only if $S$ and~$S'$ are r-isomorphic.
We start with two observations.

\begin{observation}\label{obs:step2}
      Let~$T_1$ and~$T_2$ be two shapes of a tree~$T$ such that $w(T_1)+ w(T_2)\leq w(T)$.
      Then either $T_1\subseteq T_2$ or $T_2\subseteq T_1$ or $T_1\cap T_2=\varnothing$.
\end{observation}
\begin{proof}
For~$k\in\{1,2\}$ let~$e_k$ be the edge of~$T$ associated to~$T_k$, that is,
$T_k$ is a component of~$T-e_k$. If~$e_1=e_2$, then either~$T_1=T_2$ or~$T_1\cap T_2  = \varnothing$.
Assume that $e_1\neq e_2$. Then either $e_2\in E(T_1)$ or $e_2\in E(T-T_1)$. If $e_2\in E(T-T_1)$,
then either $T_1\subseteq T_2$ or $T_2\subseteq T-T_1$ in which case $T_1\cap T_2=\varnothing$.
If $e_2\in E(T_1)$, then $T_2\subseteq T_1$: otherwise, $T_1\cap T_2\neq\varnothing$ and $T-T_1\subseteq T_2$,
    so that $w(T_1)+w(T_2)>w(T)$, contrary to the assumption.
\end{proof}
\begin{observation}\label{o.init}
Let~$(T,w)$ be a weighted tree such that every leaf
      has weight~$1$. Assume that
      we know the total weight~$w(T)$ of~$T$ and that,
      for each~$j\leq w(T)/2$ and each $j$-form~$F$, we know the number of
      shapes of~$(T,w)$ that belong to~$F$.  Then we know~$T$.
\end{observation}
\begin{proof} We use Observation~\ref{obs:step2}.  We order
      the shapes of~$(T,w)$ of weight at most~$w(T)/2$ decreasingly according
      to their weights.  Let~$m$ be the maximum weight of such a shape of~$T$
      and let~$S_1, \dotsc, S_a$ be the shapes with weight~$m$. Note that we
      know precisely these~$a$ trees.  In addition, either the
      shapes~$S_1,\dotsc,S_a$ are joined in~$T$ to the same vertex, or~$a=2$
      and~$m=w(T)/2$.  In the latter case ($m=w(T)/2$) we know that $T$
      consists of the two weighted rooted trees~$S_1$ and~$S_2$ (each of
      weight~$m$) with an edge between their roots: this ends the proof for this case.
      Assume that $m<w(T)/2$. We let~$r$ be the additional vertex to which we link
      each of~$S_1,\dotsc,S_a$.

      We show by descending induction on~$j\in\{2,\dotsc,m\}$ that we know the
      subtree of~$T$ induced by all shapes of~$T$ with weight
      in~$\{j,\dotsc,\lfloor W(T)/2\rfloor\}$. The induction has thus been
      initialized above, so assume that~$j\leq m-1$. Let~$S_1,\dotsc,S_t$ be
      the shapes of~$T$ with weight in~$\{j+1,\dotsc,\lfloor W(T)/2\rfloor\}$.
      Note that we know, in particular, each of these~$t$ trees.  The shapes
      of~$T$ of weight equal to~$j$, if any, are either shapes of~$S_1, \dotsc,
      S_t$ or joined to~$r$ by an edge from their root. Fix a~$j$-form~$F$.
      Since we do know the total number of shapes belonging to~$F$ and
      contained in each of~$S_1,\dotsc,S_t$ (because we know precisely those
      subtrees), we can deduce the number of shapes that belong to~$F$ and are
      attached to~$r$. As this argument applies to all~$j$-forms~$F$, we infer
      that we know the subtree of~$T$ formed by all shapes with weight
      contained in~$\{j,\dotsc,\lfloor w(T)/2\rfloor\}$. The reconstruction
      of~$T$ is almost finished: letting $w_0$ be the total weight of the tree
      we built so far, it only remains to add $w(T)-w_0$ new leaves, each joined
      to the vertex~$r$.  This concludes the proof.
\end{proof}

Let~$(T,w)$ be a weighted tree.
Let~$\alpha(T)=(\alpha_1,\dotsc,\alpha_n)$ be the weights of the shapes of~$T$,
with $\alpha_1<\dotsb<\alpha_n$. The definition of a shape implies that
$\alpha_1\ge 2$.  

We shall consider \emph{connected partitions} of the tree~$T$, \emph{i.e.}, partitions
of the vertex set of~$T$ into connected subsets. Later in the paper we refer to
connected partitions of~$T$ simply as partitions of~$T$.  We shall also
consider the partitions of the integer~$w(T)$.  To distinguish between them
clearly, partitions of an integer are referred to as \emph{expressions}.  For
each partition~$P$ of~$T$, the weights of the parts of~$T$ form an expression
of~$w(T)$, which we call the \emph{characteristic} of~$P$.
\begin{itemize}
      \item A \emph{$j$-expression of an integer~$m$} is a partition of~$m$ where
          one of the parts is equal to~$m-j$.
       \item For~$i\in\{1,\dotsc,\ell\}$, let~$m_i$ be an integer and~$E_i$ an expression of~$m_i$.
             We define~$[E_1,\dotsc, E_\ell]$ to be the expression
            of~$\sum_{i=1}^\ell m_i$ equal to the concatenation of~$E_1,\dotsc,E_{\ell}$.
          In particular, if $S$ is a shape of~$T$ with weight~$\alpha_j$,
          then $[\Ex(S),w(T)-\a_j]$ is an $\a_j$-expression of~$w(T)$.
    \item A \emph{$j$-partition} of~$T$ is a partition of~$T$ whose
       characteristic is a $j$-expression of~$w(T)$. In other words,
            one of the components of the partition has weight~$w(T)-j$.
    \item A $j$-partition~$(T_0,\dotsc,T_k)$ of~$T$ with~$w(T_0)=w(T)-j$ is
          \emph{shaped} if there exists an edge~$e$ of~$T$ such that $T_0$ is
          one of the components of~$T-e$. Any such edge~$e$ is then \emph{associated}
          to~$(T_0,\dotsc,T_k)$.
    \item If $S$ is a shape of~$T$ with weight~$\alpha_j$ and vertex
          set~$V(S)=\{v_1,\dotsc,v_s\}$, we define $P(S)$ to be~$(V(T)\setminus
          V(S),\{v_1\},\dotsc,\{v_s\})$, which is a shaped $\a_j$-partition of~$T$.
\end{itemize}
For an expression~$E$ of a positive integer, we let~$\U(T,w,E)$ be the number
of partitions of~$(T,w)$ with characteristic~$E$. Note that this number is~$0$
if $E$ is not an expression of~$w(T)$.  We note that there is a bijection between connected partitions and 
edge subsets given by taking all edges of $T$ joining two vertices in different blocks of the connected partition
and thence~$\U(T,w,E)$ turns out to be the coefficient of $x_E$ in the W-polynomial of $(T,w)$.

We note
that among the partitions of~$T$  corresponding to a given expression, some are
shaped and others are not. If all the vertex weights are equal to one, we abbreviate~$\U(T,w,E)$
as~$\U(T,E)$. The proofs of both theorems rely on the following procedure.
\begin{procedure}\label{proc:1}\mbox{}\\
\textsc{input:} The polynomial~$W_{(T,w)}$; an
      integer~$j\in\{\a_2,\dotsc,\a_{\ell}\}$, where~$\ell$
      is the least integer~$i$ such that $\a_i> w(T)/2$; a $j$-expression~$E$ of $w(T)$ and, for
      each~$j'<j$ and each~$j'$-form~$F$, the number of shapes~$S$ of~$T$ that
      are isomorphic to a member of~$F$ (hence, according to the notation introduced above,
      possibly but not necessarily r-isomorphic, and hence not necessarily a member
      of~$F$).

\noindent
      \textsc{output:} The number of shaped $j$-partitions of~$T$ with characteristic~$[w(T)-j,E]$.
\end{procedure}
\noindent
Let us see how this procedure allows us to establish Theorem~\ref{thm.mmm}.
\begin{proof}[Proof of Theorem~\ref{thm.mmm}]
Fix two good weighted trees~$(T,w)$ and~$(T',w')$ with $W_{(T,w)}=W_{(T',w')}$.  By
Observation~\ref{o.init},  $(T,w)$ and~$(T',w')$ are isomorphic if $w(T)=
w'(T')$ and for each $j$-form~$F$ where~$j\leq w(T)/2$, the numbers of shapes
of~$T$ and of~$T'$ that belong to~$F$ are equal. To establish this, first note
that the vector $\alpha(T)=(\alpha_1,\dotsc,\alpha_n)$ can be computed
from~$W_{(T,w)}$, since the coordinates correspond to the partitions of~$T$
into two subtrees (each with at least two vertices). Thus
$\alpha(T')=\alpha(T)$.

We prove by induction on~$j\in\{\alpha_1,\dotsc,\lfloor w(T)/2\rfloor\}$ that
for every $j$-form~$F$, the numbers of shapes of~$T$ and of~$T'$ that belong
to~$F$ are the same. So suppose first, as the base case of the induction, that $j=\alpha_1$.  Recall that
$\alpha_1\ge2$.  Furthermore, a shape~$S$ of~$T$ or~$T'$ belongs to an
$\a_1$-form if and only if~$S$ is the star on~$\alpha_1$ vertices rooted at its
centre. This is because the leaves and their neighbours have weight~$1$. It
follows that the number of shapes of~$T$ of weight~$\alpha_1$ can be calculated
from~$W_{(T,w)}$ and thus this number is the same for~$(T',w')$.

Now we establish the induction step. 
For convenience, if $F$ is a $j$-form, let~$n_T(F)$ be the number of shapes
      of~$T$ that belong to~$F$; we use a similar notation for~$T'$.
Let~$j\in\{\a_1+1,\dotsc,\lfloor w(T)/2\rfloor\}$. The induction hypothesis
      is that $n_T(F')=n_{T'}(F')$ for every $j'$-form~$F'$ and every $j'<j$.
      We want to establish that
\begin{equation}\label{eq-t1}      
\text{for every $j$-form~$F$},\quad n_T(F)=n_{T'}(F).
\end{equation}
This will prove Theorem~\ref{thm.mmm} by Observation~\ref{o.init}.

We first set a partial order on the $j$-forms, which allows us to link
tree partitions with $j$-forms.  Given a $j$-form~$F$, we define~$\Ex(F)$ to
be~$\Ex(f)$ for an arbitrary representative~$f$ of~$F$.  (This definition is
valid, since all representatives of a $j$-form are r-isomorphic rooted weighted
trees.) A $j$-form $F'$ is \emph{smaller than} a $j$-form~$F$ if $\Ex(F')$ is
a proper refinement of~$\Ex(F)$. If $P=(T_0,\dotsc,T_k)$ is a shaped
$j$-partition of~$T$ where $w(T_0)=w(T)-j$, we define $S(P)$ to be the shape
of~$T$ formed by the union of all parts of~$T$ different from~$T_0$, that is,
      $S(P)\coloneqq\cup_{i=1}^k T_i=T-T_0$, rooted at the
      end-vertex of the edge associated to~$P$.

A key observation is that if~$P$ is a shaped $j$-partition of~$T$ with
characteristic~$[\Ex(F), w(T)- j]$ for some $j$-form~$F$, then $\Ex(S(P))$ is a refinement
of~$\Ex(F)$, possibly equal to~$\Ex(F)$.

We prove~\eqref{eq-t1} by induction on the $j$-form~$F$ considered (with respect to the
partial order defined above).  

We first deal with the case where~$T$ has no
shape that belongs to a $j$-form~$F'$ such that $\Ex(F')$ is a proper refinement
      of~$\Ex(F)$. We demonstrate the following assertion.

\begin{assert}\label{as:1}
If~$T$ has no $j$-form~$F'$ such that
$\Ex(F')$ is a proper refinement of~$\Ex(F)$, then the
number of shaped $j$-partitions of~$T$ with characteristic~$[\Ex(F), w(T)- j]$ is
equal to~$n_T(F)$.
\end{assert}

This assertion implies that $n_T(F)=n_{T'}(F)$ since by Procedure~\ref{proc:1} and by
the induction hypothesis,
the number of shaped $j$-partitions of~$T$ with characteristic~$[\Ex(F), w(T)- j]$ is
equal to the number of shaped $j$-partitions of~$T'$ with characteristic~$[\Ex(F), w(T)- j]$.

To establish Assertion~\ref{as:1}, we first note that
each shape of~$T$ that belongs to~$F$ provides exactly one shaped $j$-partition of~$T$
with characteristic~$[\Ex(F), w(T)- j]$. On the other hand,
if $P$ is a shaped $j$-partition of~$T$ with
characteristic~$[\Ex(F), w(T)- j]$, then $\Ex(S(P))$ is a refinement of~$\Ex(F)$, which by our
hypothesis on~$F$ must be equal to~$\Ex(F)$. So~$S(P)$ gives rise to precisely
one shaped $j$-partition of~$T$ with characteristic~$[\Ex(F), w(T)- j]$, namely~$P$.
As $\Ex(F)=\Ex(S(P))$, it follows from Definition~\ref{def.good}\ref{good-weishape} that~$S(P)$ belongs to~$F$,
which ends the proof of Assertion~\ref{as:1}.

In the induction step we assume that $n_T(F')=n_{T'}(F')$ for every
$j$-form~$F'$ such that $\Ex(F')$ is a proper refinement of~$\Ex(F)$. Observe that
for each $j$-form~$F'$ with $F'<F$, each shape of~$T$ that belongs to~$F'$
gives rise to a certain number of shaped $j$-partitions of~$T$ with
characteristic~$\Ex(F)$, and this number depends only on~$F'$. Thus the number~$n'_T(F)$
of shaped $j$-partitions of~$T$ with characteristic~$[\Ex(F), w(T)- j]$ such that
$\Ex(S(P))$ is a proper refinement of $\Ex(F)$ depends only on the multi-set~$\sst{n_T(F')}{F'<F}$. As
$\sst{n_T(F')}{F'<F}=\sst{n_{T'}(F')}{F'<F}$, we deduce that $n'_T(F)= n'_{T'}(F)$.
We demonstrate the following assertion.

\begin{assert}\label{as:2}
The number of shaped $j$-partitions of~$T$ with characteristic~$[\Ex(F), w(T)- j]$ is
equal to~$n'_T(F)+ n_T(F)$.
\end{assert}

This assertion follows analogously as Assertion~\ref{as:1}.
Moreover, we established in the paragraph above that $n'(T,F)= n'(T',F)$.
Since the number of shaped partitions of~$T$ with characteristic~$[Ex(F),w(T)-j]$ is equal
to the number of shaped partitions of~$T'$ with characteristic~$[Ex(F),w(T)-j]$
by Procedure~\ref{proc:1}, we deduce that $n_T(F)= n_{T'}(F)$ by Assertion~\ref{as:2}.
This establishes~\eqref{eq-t1}, and hence finishes the proof of Theorem~\ref{thm.mmm}.
\end{proof}

As we see next, the notion of a shape and Procedure~\ref{proc:1} turn out to be
essential tools to study Stanley's isomorphism conjecture restricted to caterpillars.

\section{Caterpillars}\label{sec:cater} 
We first observe that Theorem~\ref{thm.mmm1} is true for all caterpillars with at most two vertices. Hence we will assume that a caterpillar has at least three vertices in this section, and we only consider weights to be~$1$; since there is then no risk
of confusion, we abbreviate~$\abs{V(T)}$ to~$\abs{T}$ for every tree~$T$.  Let~$T$
be a caterpillar (with at least three vertices).  The \emph{spine} of~$T$ is the
unique path~$P$ of~$T$ such that every leaf of~$T$ is at distance exactly one from
a vertex of~$P$.

Before proving Theorem~\ref{thm.mmm1}, we formalize a simple but crucial
observation, which is used repeatedly and implicitly in the proof of
Theorem~\ref{thm.mmm1}.  
\begin{observation}\label{obs:scru} Every shape of a caterpillar~$T$ is rooted at
      a vertex of the spine of~$T$.
\end{observation}
 It follows from~Observation~\ref{obs:scru} that for every
      integer~$j$, the number of shapes of~$T$ with~$j$ vertices belongs to~$\{0,1,2\}$.

\noindent   If $T$ is a caterpillar, and $E$ is an expression of~$j$ so that no
part of~$E$ is equal to~$\abs T-j$, then we define~$\U_s(T,E)$ to be the number of
\textbf{shaped} $j$-partitions of~$T$ with characteristic~$[\abs{T}-j, E]$.  
 Let~$S_k$ be the star on~$k$
      vertices --- thus $S_1$ is a single vertex.  We always consider a star to be
      rooted at its center.
      If $T$ is a rooted tree then we define~$S_k\rightarrow T$ to be the tree
      rooted at the center of~$S_k$ and obtained by joining the root of~$T$ to
      that of~$S_k$ by an edge.  Hence if $T$ is a rooted caterpillar, then
      $S_k\rightarrow T$ is also a rooted caterpillar. 
            
      \medskip\noindent
      
Let~$\A$ be the collection of rooted caterpillars~$A$ such that
\begin{itemize}
      \item $A$ is a single vertex; or
      \item $A$ is a rooted edge; or
      \item $\abs{A}\ge3$ and the root of~$A$ is either an end-vertex of the spine or
            a leaf attached to an end-vertex of the spine.
\end{itemize}
If $A\in \A$ then the \emph{reverse~$\tilde{A}$} of~$A$ is defined as follows. If $A$ is a single
vertex then $\tilde{A}\coloneqq A$. If $A$ is a rooted edge then $\tilde{A}$ is the same edge
rooted at the other end-vertex.  If $A$
has at least three vertices and the root is an end-vertex of the spine then $\tilde{A}$
is obtained from~$A$ by resetting the root at the other end-vertex of the spine. If
$A$ has at least three vertices and the root is  a leaf attached to an end-vertex
of the spine then $\tilde{A}$ is obtained from~$A$ by resetting the root at an arbitrary
leaf attached to the other end-vertex of the spine. (We note that such a leaf always
exists by the definition of the spine.)

\medskip
 
\begin{observation}\label{obs:new1}
      Let $A, B\in \A$ such that~$A$ and~$B$ are isomorphic but not r-isomorphic.
Let~$o$, $o_1$ and~$o_2$ be positive integers.
\begin{enumerate}
      \item The caterpillars~$S_o\rightarrow A$ and~$S_o\rightarrow B$ are not
            isomorphic; and\label{it:new1}
      \item neither are the caterpillars~$S_{o_2}\rightarrow S_{o_1}\rightarrow A$
            and~$S_{o_2}\rightarrow S_{o_1}\rightarrow B$.\label{it:new2}
\end{enumerate}
\end{observation}
\begin{proof}
The statements are vacuously true if $\abs{A}\le2$, so we assume that $A$ has at least three
vertices --- and thus so has~$B$. Given an element $C\in\A$ with~$\abs{C}\ge3$, we
let~$r_C$ be the root of~$C$ and we define the degree sequence~$s_C$ of~$C$ as
follows.~Let $w_1\dotso w_t$ be the spine of~$C$, where~$w_1$ is closest to~$r_C$.
The degree sequence of~$C$ is $s_C\coloneqq(\deg(w_1),\dotsc,\deg(w_t))$.  The
\emph{reverse} of~$s_C$ is then the sequence~$(\deg(w_t),\dotsc,\deg(w_1))$.  We
observe that two elements~$C$ and~$C'$ of~$\A$ (with at least three vertices) are
isomorphic if and only if $s_C=s_{C'}$ or $s_{C'}$ is the
reverse of~$s_C$. Furthermore,~$C$ and~$C'$ are r-isomorphic
if and only if $s_C=s_{C'}$ and $\deg(r_C)=\deg(r_{C'})$ (that is, either both roots
have degree one, or both roots have degree greater than one).

Let us make another preliminary remark. If $\deg_A(r_A)=1\neq\deg_B(r_B)$, then in
each of~\eqref{it:new1} and~\eqref{it:new2} the caterpillars obtained from~$A$ and
from~$B$ have spines of different lengths, so they are not isomorphic.  We can
thus assume that either both of~$r_A$ and~$r_B$ have degree one, or both have
degree greater than one.  This implies that $s_A\neq s_B$ and~$t>1$, as otherwise $A$
and~$B$ would be r-isomorphic. Consequently, $s_B$ is the reverse
of~$s_A$. Let us write~$s_A=(a_1,\dotsc,a_t)$.

\eqref{it:new1}. For convenience, set~$A'\coloneqq S_o\rightarrow A$
and~$B'\coloneqq S_o\rightarrow B$.  We know that $s_B=(a_t,\dotsc,a_1)\neq s_A$.
Suppose first that $\deg_A(r_A)=1=\deg_B(r_B)$. Then $s_{A'}=(o,2,a_1,\dotsc,a_t)$
if $o>1$ while $s_{A'}=(2,a_1,\dotsc,a_t)$ if $o=1$.  Similarly,
$s_{B'}=(o,2,a_t,\dotsc,a_1)$ if $o>1$ while $s_{B'}=(2,a_t,\dotsc,a_1)$ if $o=1$.
In either case, we see that $s_{A'}\neq s_{B'}$ as $s_A\neq s_B$.  So suppose for
a contradiction that $s_{B'}$ is the reverse of~$s_{A'}$.  In the former case,
\emph{i.e.},~$o>1$, this means that $(o,2,a_1,\dotsc,a_t)=(a_1,\dotsc,a_t,2,o)$.  Then
$a_j=o$ for~$j$ odd and $a_j=2$ for~$j$ even. In
addition, $a_t=o$ and $a_{t-1}=2$, showing that~$t$ must be odd unless~$o=2$.
However, either way this yields that $s_A=s_B$, a contradiction.  In the latter
    case, \emph{i.e.},~$o=1$, we have $(2,a_1,\dotsc,a_t)=(a_1,\dotsc,a_t,2)$, so~$a_i=2$ for
each~$i\in\{1,\dotsc,t\}$ which again contradicts that $s_A\neq s_B$.

It remains to deal with the case where $\deg_A(r_A)\neq1\neq\deg_B(r_B)$.
    If~$o>1$, then $s_{A'}=(o,1+a_1,a_2,\dotsc,a_t)$
    and~$s_{B'}=(o,1+a_t,a_{t-1},\dotsc,a_1)$.  If~$o=1$, then
    $s_{A'}=(1+a_{1},a_2,\dotsc,a_t)$ and~$s_{B'}=(1+a_t,a_{t-1},\dotsc,a_1)$.
    In either case, note that $s_{A'}\neq s_{B'}$ because $s_A\neq s_B$.
    Further, if~$s_{B'}$ is the reverse of~$s_{A'}$, then it implies that
    $o>1$, $a_t=o=a_1$ and $a_i=o+1$ for~$i\in\{2,\dotsc,t-1\}$, leading
    to~$s_A=s_B$, a contradiction.  This ends the proof of~\eqref{it:new1}.

\eqref{it:new2}. For convenience, set $A'\coloneqq S_{o_2}\rightarrow
S_{o_1}\rightarrow A$ and~$B'\coloneqq S_{o_2}\rightarrow S_{o_1}\rightarrow B$.
Assume first that $\deg_A(r_A)=1=\deg_B(r_B)$. Then we infer as before that
\[
      s_{A'}=\begin{cases}
            (2,2,a_1,\dotsc,a_t)&\quad\text{if $o_1=1$ and $o_2=1$},\\
            (o_2,2,2,a_1,\dotsc,a_t)&\quad\text{if $o_1=1$ and $o_2>1$},\\
            (1+o_1,2,a_1,\dotsc,a_t)&\quad\text{if $o_1>1$ and $o_2=1$},\\
            (o_2,1+o_1,2,a_1,\dotsc,a_t)&\quad\text{if $o_1>1$ and $o_2>1$}.\\
      \end{cases}
\]
and
\[
      s_{B'}=\begin{cases}
            (2,2,a_t,\dotsc,a_1)&\quad\text{if $o_1=1$ and $o_2=1$},\\
            (o_2,2,2,a_t,\dotsc,a_1)&\quad\text{if $o_1=1$ and $o_2>1$},\\
            (1+o_1,2,a_t,\dotsc,a_1)&\quad\text{if $o_1>1$ and $o_2=1$},\\
            (o_2,1+o_1,2,a_t,\dotsc,a_1)&\quad\text{if $o_1>1$ and $o_2>1$}.\\
      \end{cases}
\]
We see that in each of the four possible cases $s_{A'}\neq s_{B'}$ as $s_A\neq
s_B$. In addition, in none of these fours cases can $s_{B'}$ be the reverse
of~$s_{A'}$, showing that $A'$ and~$B'$ are not isomorphic. For instance, in the
second case it would imply that $t$ is~$1$ modulo~$3$ and $a_i=o_2$ if $i$ is
equal to~$1$ modulo~$3$, while $a_i=2$ otherwise; however this
would yield that $s_A=s_B$, a contradiction.  To check the fourth case, it is
useful to consider the value of~$t$ modulo~$3$.

It remains to deal with the case where $\deg_A(r_A)\neq1\neq\deg_B(r_B)$.  We
infer the following expressions.
\[
      s_{A'}=\begin{cases}
            (2,1+a_1,a_2,\dotsc,a_t)&\quad\text{if $o_1=1$ and $o_2=1$},\\
            (o_2,2,1+a_1,a_2,\dotsc,a_t)&\quad\text{if $o_1=1$ and $o_2>1$},\\
            (1+o_1,1+a_1,a_2,\dotsc,a_t)&\quad\text{if $o_1>1$ and $o_2=1$},\\
            (o_2,1+o_1,1+a_1,a_2,\dotsc,a_t)&\quad\text{if $o_1>1$ and $o_2>1$}.\\
      \end{cases}
\]
and
\[
      s_{B'}=\begin{cases}
            (2,1+a_t,a_{t-1},\dotsc,a_1)&\quad\text{if $o_1=1$ and $o_2=1$},\\
            (o_2,2,1+a_t,a_{t-1},\dotsc,a_1)&\quad\text{if $o_1=1$ and $o_2>1$},\\
            (1+o_1,1+a_t,a_{t-1},\dotsc,a_1)&\quad\text{if $o_1>1$ and $o_2=1$},\\
            (o_2,1+o_1,1+a_t,a_{t-1},\dotsc,a_1)&\quad\text{if $o_1>1$ and $o_2>1$}.\\
      \end{cases}
\]
It follows that in none of the four cases the sequence~$s_{B'}$ ie equal
to~$s_{A'}$ or to the reverse of~$s_{A'}$, again relying on the fact that $s_A\neq
s_B$.
\end{proof}

We are
now ready to proceed with the proof of~Theorem~\ref{thm.mmm1}.

\begin{proof}[Proof of Theorem~\ref{thm.mmm1}]      Let~$T$ be a caterpillar. We proceed by induction on the number of vertices of~$T$,
      the theorem being true if~$\abs{T}<4$. We now deal with the inductive step.
      As before, we note that the vector $\alpha(T)=(\alpha_1,\dotsc,\alpha_n)$
      can be computed from~$U_T$, since the coordinates correspond to the
      partitions of~$T$ into two subtrees (each with at least two vertices).
      We prove by induction on~$j\in\{\alpha_1,\dotsc,\lfloor \abs{T}/2\rfloor\}$
      that for every $j$-form~$F$, we can deduce from~$U_T$ the number
      of shapes of~$T$ that belong to~$F$. Observation~\ref{o.init} ensures
      then that we can reconstruct~$T$.  Analogously as in a previous proof
      the number of shapes of~$T$ of size~$\alpha_1$ can be calculated
      from~$U_T$. This number is one or two since $T$ is a caterpillar. 

      We proceed inductively
      and, at each step of the inductive process, we update our knowledge of the
      two ends of~$T$, by increasing the size of our knowledge of (at least) one
      end of~$T$. It is important to note that to know the number of shapes of~$T$
      that belong to a given $j$-form~$F$ for some~$j\ge2$, it is enough to know
      both ends of~$T$ of order~$j$. At any given step, we let~$R_1$ and~$R_2$ be
      the currently known forms of the two ends of~$T$.  Hence after the first
      step $R_1= S_{\alpha_1}$ and $R_2=\varnothing$ or $R_2=R_1$, depending on
      whether~$\U(T,[\abs{T}-\alpha_1,\alpha_1])$ equals~$1$ or~$2$. (As reported
      earlier, this number can be deduced from the $U$-polynomial of~$T$.)

      Let~$j\in\{\a_1+1,\dotsc,\lfloor \abs{T}/2\rfloor\}$.  We assume that for
      each~$j'\in\{\a_1,\dotsc,j-1\}$ and each $j'$-form~$F$ we know the number
      of shapes of~$T$ that belong to~$F$. Let us establish this last statement
      for~$j'=j$. If $j\notin\{\a_{2},\dotsc,\a_n\}$, then we know that the sought
      number is~$0$, by the definition of~$(\a_1,\dotsc,\a_n)$. So we suppose
      now that $j=\a_k$ for some integer~$k\in\{2,\dotsc,n\}$.  We set
      $m\coloneqq\alpha_k- \alpha_{k-1}$. (Recall that this number
      can be deduced from the $U$-polynomial.)  
      Let $\alpha_{k-1}=\abs{R_1}\geq \abs{R_2}$, with~$R_2$ possibly empty. Set $p\coloneqq\a_k- \abs{R_2}$,
      let $R'_1\coloneqq S_m \rightarrow R_1$ and $R'_2\coloneqq S_p \rightarrow R_2$.
      
         If $R_1$ and $R_2$ are r-isomorphic and $\alpha_k= 1$ then we set~$R_1\coloneqq R'_1$
and leave $R_2$ unchanged.  If $R_1$ and $R_2$ are r-isomorphic and $\alpha_k= 2$ then we set~$R_1\coloneqq R'_1$
and~$R_2\coloneqq R'_2$. Hence from now on we assume that~$R_1$ and~$R_2$ are not r-isomorphic.
We distinguish three cases.

\medskip

\textbf{[(1)] Let~$T$ have two $\a_k$-shapes.}

\noindent
Then we update both~$R_1$ and~$R_2$, that is, we set~$R_1\coloneqq R'_1$
and~$R_2\coloneqq R'_2$.

\medskip

    \textbf{[(2)]  Let~$T$ have exactly one $\alpha_k$-shape, \emph{i.e.}, either~$R'_1$
      or~$R'_2$. Moreover let~$R'_1$ and~$R'_2$ be not isomorphic.}

\noindent
We recall that $\a_k\leq \abs T /2$.
As $\abs{R_i'}<\abs{T}$, we know by induction that $U_{R_1'}\neq U_{R_2'}$. Hence
there is an expression~$E'$ of~$\abs{R_1'}=\alpha_k$ such that $r_1\coloneqq
\U(R'_1,E')\neq r_2\coloneqq \U(R_2',E')$. 

Now comes an important observation that will be used repeatedly in this proof:
we know there is only one $\a_k$-shape in $T$, and thus all shaped $\a_k$-partitions of $T$
have to come from partitions where one removes the edge associated to this shape and any subset of edges inside this shape.

Therefore, there is
a unique~$i\in\{1,2\}$ such that $\U_s(T,E')=r_i$ and we can
determine it by Procedure~\ref{proc:1}. We set~$R_i\coloneqq R_i'$ and leave~$R_{3-i}$
unchanged.

\medskip

\textbf{[(3)] Let~$T$ have exactly one $\alpha_k$-shape, \emph{i.e.}, either~$R'_1$
or~$R'_2$. Moreover let~$R'_1$ and~$R'_2$ be isomorphic but not r-isomorphic.}

\noindent
In this case we explicitly know the unique isomorphism class for the  $\alpha_k$-shapes of $T$. Therefore we know, for each 
$\alpha_k$-form~$F$, the number of shapes of~$T$ that are
isomorphic (but not necessarily r-isomorphic) to a member of~$F$.  We
observe that $k<n$. We set~$q\coloneqq \a_{k+1}- \a_k$.  

By Procedure~\ref{proc:1}, we know for each $\alpha_{k+1}$-expression~$E$ the number of
shaped $\a_{k+1}$-partitions of~$T$ with characteristic~$E$.  

\medskip

There are four candidates for an $\a_{k+1}$-shape of~$T$, namely~$S_{1,1}\coloneqq
S_q\rightarrow S_m\rightarrow R_1=S_q\rightarrow R_1'$, $S_{2,1}\coloneqq
S_{q+m}\rightarrow R_1$, $S_{1,2}\coloneqq S_{q+p}\rightarrow R_2$
and~$S_{2,2}\coloneqq S_q\rightarrow S_p\rightarrow R_2=S_q\rightarrow R_2'$. 

We now introduce some labels for the vertices of the stars~$S_q$, $S_{q+m}$
and~$S_{q+p}$. The vertices of~$S_q$ are labelled~$v_1,\dotsc,v_q$, where~$v_q$
is the centre of~$S_q$. By extension, the corresponding vertices of~$S_{1,1}$
and~$S_{2,2}$ inherit those labels. For~$i\in\{p,m\}$, the vertices of~$S_{q+i}$
are labelled~$v_1,\dotsc,v_{q+i}$ where, this time, the labels~$v_1,\dotsc,v_q$
are assigned to leaves only. Similarly, the corresponding vertices of~$S_{1,2}$
and of~$S_{2,1}$ inherit those labels.
Thus, for example, the vertex~$v_q$ of~$S_{1,1}$ is the centre
of the star~$S_q$ and hence the root of~$S_{1,1}$, while the vertex~$v_q$ of~$S_{1,2}$
is one of the leaves of~$S_{q+p}$ and hence is adjacent to the root of~$S_{1,2}$.

\medskip

\textbf{[(3.1)] Let~$T$ have two $\a_{k+1}$-shapes.} 

There are two possibilities for the two $\a_{k+1}$-shapes of~$T$: either~$S_{1,1}, S_{1,2}$ or~$S_{2,1}, S_{2,2}$. We note that
this implies that $\a_{k+1}\leq |T|/2$.
For~$i\in \{1,2\}$, let~$T_i$ be any caterpillar with $|T_i|= |T|$ whose  
$\alpha_{k+1}$-shapes are exactly~$S_{i,1}$ and~$S_{i,2}$.

\begin{observation}\label{obs:new2}\mbox{}
    If~$q> 1$ then~$S_{i,j}$ and~$S_{i',j'}$ are not isomorphic (as unrooted trees) whenever~$i,i',j,j' \in\{1,2\}$ with~$i\neq i'$. 
\end{observation}
\begin{proof}
Comparing the lengths of the spines, the only possible pairs
of isomorphic trees are: $S_{1,1}$ with~$S_{2,2}$, and~$S_{1,2}$ with~$S_{2,1}$.
      However, the fact that $R_1'$ and~$R_2'$
      are isomorphic but not r-isomorphic prevents each of these
pairs to consist of isomorphic trees, using
Observation~\ref{obs:new1}\eqref{it:new1} for the former one.
\end{proof}

Let~$E$ be an expression of $\a_{k+1}$. We note that each $T_i$ has exactly two vertices labelled by $v_q$, namely the root of $S_{i,i}$ and a leaf of $S_{i,3-i}$ attached to the root of $S_{i,3-i}$. 
We classify the shaped $\a_{k+1}$-partitions of~$T_i$, for each~$i \in \{1, 2\}$,
into four classes~$C(E,i,1)$, $C(E,i,2)$, $C(E,i,3)$ and~$C(E,i,4)$. To this end,
let~$\mathcal{E}_i$ be the collections of all shaped $\a_{k+1}$-partitions of~$T_i$ of characteristic~$[|T|-\a_{k+1},E]$.
An element~$P$ of~$\mathcal{E}_i$ \emph{partitions} a subtree~$G$ of~$T_i$ if a subset (possibly of order one)
of the parts of~$P$ forms a partition of~$G$.

\medskip\noindent

(1) We let~$C(E,i,1)$ be the collection of all elements of~$\mathcal{E}_i$ such that a subset of parts of $P$ is 
a partition of the unique $\a_k$-shape of~$T_i$. 

(2) We let~$C(E,i,2)$ be the collection of all elements~$P\in\mathcal{E}_i\setminus C(E,i,1)$
such that $\{v_q\}\subset V(S_{i,i})\cup V(S_{i,3-i})$ is not a part of~$P$.

(3)  We let~$C(E,i,3)$ be the collection of all elements~$P\in\mathcal{E}_i\setminus C(E,i,1)$
such that $\{v_i\}\subseteq V(S_{i,3-i})$ is  a part of~$P$ for each~$i\leq q$.

(4)  We let~$C(E,i,4)$ be the collection of all elements~$P\in\mathcal{E}_i\setminus C(E,i,1)$
such that $\{v_q\}\subseteq V(S_{i,3-i})$ is a part of~$P$ and there exists~$\ell\in\{1,\dotsc,q-1\}$
such that~$\{v_\ell\}\subseteq V(S_{i,3-i})$ is not a part of~$P$.

\medskip\noindent

\begin{observation}\label{o.kk}
Let~$i \in \{1, 2\}$ and let~$E$ be an expression of $\a_{k+1}$.
    
\begin{enumerate}
\item[(1)] The partitions in~$C(E,i,1)$ partition the shape~$S_{i,i}$ of~$T_i$.
      Moreover, there is a bijection~$F$ from~$C(E,1,1)$ to~$C(E,2,1)$ so that
      for each~$P$, there is a bijection between the sets of components of~$P$
        and~$F(P)$ that identifies the class of~$P$ containing the root of the
        $(\a_{k+1})$-shape with the class of~$F(P)$ containing the root of the
        $(\a_{k+1})$-shape.

 \item[(2)]  There is a bijection~$F$ from~$C(E,1,2)$ to~$C(E,2,2)$ so that
     if~$P$ partitions the shape~$S_{i,j}$ of~$T_i$, then~$F(P)$ partitions the
        shape~$S_{3-i,j}$ of~$T_{3-i}$ and there is a bijection between the
        sets of components of~$P$ and~$F(P)$ that identifies the class of~$P$
        containing the root of the $(\a_{k+1})$-shape with the class of~$F(P)$
        containing the root of the $(\a_{k+1})$-shape.

\item[(3)] The partitions in~$C(E,i,3)$ partition the shape~$S_{i,3-i}$ of~$T_i$.
    Moreover, there is a bijection~$F$ from~$C(E,1,3)$ to~$C(E,2,3)$ so that
        for each~$P$, there is a bijection between the sets of components
        of~$P$ and~$F(P)$ that identifies the class of~$P$ containing the root
        of the $(\a_{k+1})$-shape with the class of~$F(P)$ containing the root
        of the $(\a_{k+1})$-shape.

\item[(4)] The partitions in~$C(E,i,4)$ partition the shape~$S_{i,3-i}$ of~$T_i$. 
\end{enumerate}
\end{observation}

\begin{proof}
Items~(2) and~(4) follow directly from the structure of the shapes~$S_{i,j}$.
Items~(1) and~(3) follow from the assumption that~$R'_1$ and~$R'_2$ are isomorphic.
\end{proof}

Let~$S^2\coloneqq S_{q+m-1}\rightarrow R_1$ and~$S^1\coloneqq S_{q+p-1}\rightarrow R_2$. 
We observe that if~$q>1$ then $S^1$ and~$S^2$ are not isomorphic since none of
the pairs~$(R'_1,R'_2)$ and~$(R_1,R_2)$ consists of r-isomorphic trees and, in addition,
$R'_1$ and~$R'_2$ are isomorphic.

\begin{observation}\label{o.kk1}
      Suppose that~$q> 1$ and let~$E$ be an expression of $\a_{k+1}-1$ such that
      $r_1\coloneqq \U(S^1,E)\neq \U(S^2,E)\eqqcolon r_2$.
      Such an expression~$E$ exists by the induction assumption since $\a_{k+1}-1< |T|$.
      Let~$i\in\{1,2\}$ such that $r_i>r_{3-i}$.
      Then $\U_s(T_i, [E, 1])> \U_s(T_{3-i}, [E, 1])$.
\end{observation}
\begin{proof}
      Let~$E'\coloneqq [E, 1]$.
      By Observation~\ref{o.kk} it suffices to show that $|C(E',i,4)|> |C(E',3-i,4)|$, which
      can be argued as follows. 

      We first observe that $|C(E',j,4)|= r_j - |C(E',j,3)|$ for each~$j\in\{1,2\}$.
      Further, $|C(E',1,3)|= |C(E',2,3)|$ by Observation~\ref{o.kk}.
      As we assumed that~$r_i> r_{3-i}$, the observation holds.
\end{proof}

\textbf{[(3.1.1)] Let $q> 1$.} Let~$E$ be the expression from Observation~\ref{o.kk1}. 
We recall that by Procedure~\ref{proc:1}, we know for each $\alpha_{k+1}$-expression~$E$ the number of
shaped $\a_{k+1}$-partitions of~$T$ with characteristic~$E$.  
Hence we know $\U_s(T, [E, 1])$ and also $\U_s(T, [E, 1])\in \{ \U_s(T_1, [E, 1]), \U_s(T_2, [E, 1])\}$, which contains
two values. Hence this case is solved by Observation~\ref{o.kk1}.

\textbf{[(3.1.2)] Let~$q= 1$.} Then~$S_{i,i}$ is isomorphic but not r-isomorphic
to~$S_{3-i,i}$ for each~$i\in \{1, 2\}$, and~$S_{1,1}$ is not isomorphic
to~$S_{2,2}$ since $R'_1$ and~$R'_2$ are not  r-isomorphic. We observe that $k+1< n$ since
$\a_{k+1}\leq |T|/2$ and not all $\a_{k+1}$-shapes of~$T$ are stars.  

We now know all the input data of Procedure~1  for $T$ and $j= \a_{k+2}$ since for each $j'\leq \a_{k+1}$ and for each $j'-$form $F$ the number of shapes $S$ of $T_1$ that are isomorphic to a member of $F$ is equal to  the number of shapes $S$ of $T_2$ that are isomorphic to a member of $F$.

Let $q'\coloneqq
\alpha_{k+2}- \alpha_{k+1}$. There are four candidates for an $\a_{k+2}$-shape
of~$T$, namely $S'_{i,j}= S_{q'}\rightarrow S_{i,j}$ for~$(i,j)\in{\{1, 2\}}^2$. 

\begin{observation}\label{o.kk3}
    The trees~$S'_{i,j}$, for~$(i,j)\in{\{1, 2\}}^2$, are mutually non-isomorphic.
\end{observation}

\begin{proof}
      For~$S'_{1,1}$ and~$S'_{2,2}$, this
      follows from Observation~\ref{obs:new1}.  Moreover, for each~$i\in
      \{1,2\}$, we know that~$S'_{i,i}$ is isomorphic to neither
      of~$S'_{i,3-1}$ and~$S'_{3-i, i}$ because the lengths of the spines
      are different. Finally we consider~$S'_{1,2}$ and~$S'_{2,1}$. We know that the
      rooted caterpillar~$R'_1$ is the reverse of~$R'_2$. Recall the degree sequences
      of caterpillars, defined in the proof of Observation~\ref{obs:new1} on page~\pageref{obs:new1}.
      Let the degree sequence~$s_{R'_1}$ of~$R'_1$ be $(a_1, \dotsc, a_n)$ --- we know that $a_1= m$. Then
      $s_{S'_{2,1}}$ is the sequence~$s_2\coloneqq (a_1, \dotsc, a_n+1,q')$ and
      $s_{S'_{1,2}}$ is the sequence~$s_1\coloneqq (a_n, \dotsc, a_1+1,q')$. We
      observe that if~$s_1= s_2$ or if~$s_1$ is the reverse of $s_2$, then $(a_1,
      \dotsc, a_n)$ is equal to its reverse, which contradicts the assumption that
      $R'_1$ and~$R'_2$ are not r-isomorphic.
\end{proof}

If~$T$ has a unique $\alpha_{k+2}$-shape then we can determine which one of the four mutually
non-isomorphic candidates it is using the induction assumption ($\a_{k+2}< |T|$) and Procedure~\ref{proc:1},
which implies that we know for each $\alpha_{k+2}$-expression~$E$ the number of
shaped $\a_{k+2}$-partitions of~$T$ with characteristic~$E$. Hence, we assume that~$T$ has two $\alpha_{k+2}$-shapes.

There are two possibilities for the two $\a_{k+2}$-shapes of $T$: either $S'_{1,1}, S'_{1,2}$ or $S'_{2,1}, S'_{2,2}$.
For~$i\in \{1,2\}$, let~$T'_i$ be any caterpillar with $|T'_i|= |T|$ whose  
$\alpha_{k+2}$-shapes are exactly~$S'_{i,1}$ and~$S'_{i,2}$.  

\medskip\noindent

Next we proceed analogously as in case~(3.1.1).
Similarly as before,
let us label the vertices of the shape~$S_{q'}$ of~$S'_{i,j}$ by~$u_1,\dotsc,u_{q'}$
for each~$(i,j)\in{\{1,2\}}^2$, where~$u_{q'}$ is the centre of~$S_{q'}$.

Let~$i \in \{1, 2\}$ and let~$E$ be an expression of  $\a_{k+2}$.
We classify the shaped $\a_{k+2}$-partitions of~$T'_i$ into four
classes~$C'(E,i,1)$, $C'(E,i,2)$, $C'(E,i,3)$ and~$C'(E,i,4)$. 
To this end,
let~$\mathcal{E}'_i$ be the collection of all shaped $\a_{k+2}$-partitions of~$T_i$ of characteristic~$[|T|-\a_{k+2},E]$.
An element~$P$ of~$\mathcal{E}'_i$ \emph{partitions} a subtree~$G$ of~$T'_i$ if a subset (possibly of order one)
of the parts of~$P$ forms a partition of~$G$.

\medskip\noindent

(1) We let~$C'(E,i,1)$ be the collection of all elements of~$\mathcal{E}'_i$ such that a subset of parts of $P$ is 
a partition the unique $\a_k$-shape of~$T'_i$. 

(2) We let~$C'(E,i,2)$ be the collection of all elements~$P\in\mathcal{E}'_i\setminus C'(E,i,1)$
such that $\{v_q\}\subseteq V(S_{i,3-i})$ is not a part of~$P$.

(3)  We let~$C'(E,i,3)$ be the collection of all elements~$P\in\mathcal{E}'_i\setminus C'(E,i,1)$
such that $\{v_q\}\subseteq V(S_{i,3-i})$ is a part of~$P$ and~$u_{q'}\in V(S'_{i,3-i})$
does not belong to the same part of~$P$ as the root of~$S_{i,3-i}\subset S'_{i,3-i}$.

(4)  We let~$C'(E,i,4)$ be the collection of all elements~$P\in\mathcal{E}'_i\setminus C'(E,i,1)$
such that $\{v_q\}\subseteq V(S_{i,3-i})$ is a part of~$P$ and~$u_{q'}\in V(S'_{i,3-i})$ belongs
to the same part of~$P$ as the root of~$S_{i,3-i}\subset S'_{i,3-i}$.

\medskip\noindent

\begin{observation}\label{o.kk6}
Let~$i \in \{1, 2\}$ and let~$E$ be an expression of  $\a_{k+2}$.
    
\begin{enumerate}
    \item[(1)] The partitions in~$C'(E,i,1)$ partition the shape~$S'_{i,i}$.  Moreover,
        there is a bijection~$F$ from~$C'(E,1,1)$ to~$C'(E,2,1)$ so that for
        each~$P$, there is a bijection between the sets of components of~$P$
        and~$F(P)$ that identifies the class of~$P$ containing the root of the
        $(\a_{k+2})$-shape with the class of~$F(P)$ containing the root of the
        $(\a_{k+2})$-shape.
    
    \item[(2)]   There is a bijection~$F$ from~$C'(E,1,2)$ to~$C'(E,2,2)$ so
        that if~$P$ partitions the shape~$S_{i,j}$ of~$T_i$, then~$F(P)$
        partitions the shape~$S_{3-i,j}$ of~$T_{3-i}$ and there is a bijection
        between the sets of components of~$P$ and~$F(P)$ that identifies the
        class of~$P$ containing the root of the $(\a_{k+2})$-shape with the
        class of~$F(P)$ containing the root of the $(\a_{k+2})$-shape.

   \item[(3)] The partitions in~$C'(E,i,3)$ partition the shape~$S'_{i,3-i}$. Moreover,
       there is a bijection~$F$ from~$C'(E,1,3)$ to~$C'(E,2,3)$ so that for
        each~$P$, there is a bijection between the sets of components of~$P$
        and~$F(P)$ that identifies the class of~$P$ containing the root of
        the $(\a_{k+2})$-shape with the class of~$F(P)$ containing the root of
        the $(\a_{k+2})$-shape.

 \item[(4)] The partitions in~$C'(E,i,4)$ partition~$S'_{i,3-i}$. 
\end{enumerate}
\end{observation}

\begin{proof}
Analogously as in the proof of Observation~\ref{o.kk}, the Items~(2) and~(4)
follow directly from the structure of the~$S'_{i,j}$.  Items~(1) and~(3) follow
from the assumption that $R'_1$ and~$R'_2$ are isomorphic.
\end{proof}

Let~$Q^1\coloneqq S_{q'}\rightarrow R'_2$ and~$Q^2\coloneqq S_{q'}\rightarrow R'_1$. 
We note that~$Q^1$ and~$Q^2$ are not isomorphic by Observation~\ref{obs:new1}. 

\begin{observation}\label{o.kk5}
Let~$E$ be an expression of~$\a_{k+2}-1$ such that $r_1\coloneqq \U(Q^1,E)\neq \U(Q^2,E)\eqqcolon r_2$.
      Let~$i\in\{1,2\}$ such that $r_i>r_{3-i}$.
      Then $\U_s(T'_i, [E, 1])> \U_s(T'_{3-i}, [E, 1])$.
\end{observation}

\begin{proof}
Set~$E'\coloneqq [E, 1]$. By Observation~\ref{o.kk6} it suffices to show that $|C'(E',i,4)|> |C'(E',3-i,4)|$, which
      can be argued as follows.
      We first observe that  $|C'(E',i,4)|= r_i - |C'(E',i,3)|$ for each~$i\in\{1,2\}$.
      Next, Observation~\ref{o.kk6} implies that $|C'(E',1,3)|= |C(E',2,3)|$.
      Since $r_i> r_{3-i}$, the observation thus holds.
\end{proof}

We recall that by Procedure~\ref{proc:1} we know for each $\alpha_{k+2}$-expression~$E$ the number of
shaped $\a_{k+2}$-partitions of~$T$ with characteristic~$E$.  
Hence we know~$\U_s(T, [E, 1])$ and also $\U_s(T, [E, 1])\in \{ \U_s(T'_1, [E, 1]), \U_s(T'_2, [E, 1])\}$, which consists
of two values. Hence the case~(3.1.2) is solved by Observation~\ref{o.kk5}.

\medskip\noindent

\textbf{[(3.2)] Let~$T$ have a unique $\a_{k+1}$-shape.}

\noindent

Let $q>1$. Using Observation~\ref{obs:new2}, the induction assumption and
Procedure~\ref{proc:1} and considering the shaped $\a_{k+1}$-partitions of~$T$, we can
determine if the unique $\a_{k+1}$-shape of~$T$ is in the set~$\{S_{1,1},
S_{1,2}\}$ or in the set~$\{S_{2,1}, S_{2,2}\}$. In the first case the unique 
  $\a_k$-shape of $T$ is~$R'_1$, in the second case the unique  $\a_k$-shape of~$T$ is~$R'_2$.

So suppose that $q= 1$. There are two pairs of isomorphic
candidates: $S_{1,1}$ is isomorphic to~$S_{2,1}$ and~$S_{1,2}$
is isomorphic to~$S_{2,2}$. We observe that for each pair, its two elements differ in the number
of leaves different from the root. Moreover, $S_{1,1}$ and~$S_{2,2}$ are not
isomorphic.  By considering the shaped $\a_{k+1}$-partitions of~$T$ we can determine
to which pair the unique $\a_{k+1}$-shape of~$T$ belongs.  We may assume, without
loss of generality, that it belongs to~$\{S_{1,1}, S_{2,1}\}$.
We now show that we can determine the number of leaves of the unique
$\a_{k+1}$-shape of~$T$ different from the root and therefore determine whether
the correct shape is~$\{S_{1,1}$ or~$S_{2,1}\}$.

We observe that $n\neq k+1$ since $q= 1$.  Since we know the isomorphism class
of the unique $\a_{k+1}$-shape of~$T$, we can determine the number of shaped
$\a_{k+2}$-partitions of~$T$ by Procedure~\ref{proc:1}.

We have
\[
\theta(T,\abs{T}-\a_{k+1}-1, \a_{k+1}, 1)= \U_s(T, \a_{k+1}, 1)+ d(T, \a_{k+1}, 1),
\]
where~$d(T, \a_{k+1}, 1)$ is
equal to the number of leaves of~$T$ outside of the unique $\a_{k+1}$-shape.
The considerations above imply that we can determine~$d(T, \a_{k+1}, 1)$. Since
we know the number of leaves of~$T$, we can also determine the number of leaves of
the  unique $\a_{k+1}$-shape of~$T$ that are different from the root. Hence we can
determine whether this shape is~$S_{1,1}$ or~$S_{1,2}$.  This finishes case~(3.2)
and thus case~(3).  

This ends our updating process and the inductive step of our induction.
Consequently, we established that we know, for
each~$j\in\{\a_1,\dotsc,\abs{T}/2\}$ and each~$j$-form~$F$, the number of shapes
of~$T$ that belongs to~$F$. Therefore Observation~\ref{o.init} ensures that we
know~$T$. This concludes the induction on the size of~$T$ and thus the proof of
Theorem~\ref{thm.mmm1}.
\end{proof}

\section{Designing Procedure~\ref{proc:1}}\label{sec:proc1} 
A \emph{$j$-situation $\sigma$} is
a sequence~$((\sigma_1,w_1),\dotsc,(\sigma_{t(\sigma)},w_{t(\sigma)}))$ of
the representatives of isomorphism classes of weighted non-rooted trees such that $t(\sigma)\ge2$, 
$\sum_{i=1}^{t(\sigma)}w_i(\sigma_i)=j$ and there are numbers $p(\sigma), q(0), q(1), \dotsc, q(p(\sigma))$ such that 

\begin{enumerate}
    \item $1\leq p(\sigma)\leq t(\sigma)$ and
$0= q(0)< q(1)< \dotsb < q(p(\sigma))= t(\sigma)$;
\item for each~$i\in\{0, \dotsc, p(\sigma)-1\}$ the weighted trees 
$(\sigma_{q(i)+s}, w_{q(i)+s})$ for~$s\in\{1, \dotsc, q(i+1)- q(i)\}$ are 
the same; and
\item if $k\notin \{q(i)+1, \dotsc, q(i+1)\}$ then the weighted tree~$(\sigma_k, w_k)$ is not isomorphic to $(\sigma_{q(i)+1}. w_{q(i)+1})$.
\end{enumerate}

A $j$-situation $\sigma$
is said to \emph{occur} in a tree~$T$ if there exists a subtree~$T'$ of~$T$ and~$t(\sigma)$ distinct edges~$e_1,\dotsc,e_{t(\sigma)}$ with exactly one end
in~$V(T')$ such that, for each~$i\in\{1,\dotsc,t(\sigma)\}$, there is an
isomorphism (thus preserving the weights \emph{but not necessarily the roots})
between~$\sigma_i$ and the component of~$T-e_i$ not containing~$T'$. Note that if
$\sigma$ occurs in~$T$, then for each~$i\in\{1,\dotsc,t(\sigma)\}$ such that
$\sigma_i$ is not a single vertex the tree~$T$ has a shape isomorphic (but not necessarily r-isomorphic)
to~$\sigma_i$.

We proceed in two steps, the first one being an exhaustive listing that depends
only on~$j$.

\noindent
\textbf{Step 1.} Explicitly list all $j$-situations for~$j\leq \alpha_l$.

\noindent
\textbf{Step 2.} For each $j$-situation~$\sigma$ from Step~1,
compute the number~$m_T(\sigma)$ of times $\sigma$ occurs in~$T$.

Before designing Step 2, we show how Steps~1 and~2 accomplish
Procedure~\ref{proc:1}.  Suppose that the two steps are completed.
Let~$E=\{w(T)- j,E_1,\dotsc,E_k\}$ be a $j$-expression of~$w(T)$.

For each $j$-situation~$\sigma=((\sigma_1,w_1),\dotsc,(\sigma_{t(\sigma)},w_{t(\sigma)}))$, let~$\Psi_{\sigma}$ be the collection of all surjections from the expression~$\{E_1,\dotsc,E_k\}$ to~$\{\sigma_1,\dotsc,\sigma_{t(\sigma)}\}$.
Two elements~$f$ and~$g$ of~$\Psi_{\sigma}$ are \emph{equivalent} if
the multi-set~$f^{-1}(\sigma_i)$ is equal to the multi-set~$g^{-1}(\sigma_i)$
for every~$i\in\{1,\dotsc,t(\sigma)\}$. We consider the equivalence classes for this relation
on~$\Psi_{\sigma}$ and we form~$\Psi_{\sigma}'$ by arbitrarily choosing one representative
in each equivalent class.
We observe
that the number~$X$ of non-shaped $j$-partitions of~$T$ with characteristic~$E$ is
\begin{equation}
    {[p(\sigma)!\prod_{1\le i\leq p(\sigma)}(q(i)- q(i-1))!]}^{-1}\sum_{\text{$j$-situation
$\sigma$}}m_T(\sigma)\sum_{f\in\Psi_{\sigma}'}
\sum_{i=1}^{t(\sigma)}\U(\sigma_i,w_i,f^{-1}(\sigma_i)),\label{eq:nshaped}
\end{equation}
where the multi-set~$f^{-1}(\sigma_i)$ is naturally interpreted as an expression.  Indeed,
a non-shaped partition of~$T$ with characteristic~$E$ corresponds precisely to
the occurrence of some
$j$-situation~$\sigma=((\sigma_1,w_1),\dotsc,(\sigma_{t(\sigma)},w_{t(\sigma)}))$
where the trees $\sigma_1\,\dotsc,\sigma_{t(\sigma)}$ are also partitioned (possibly
trivially).  Recalling that $\U(\sigma_i,w_i,E')$ is zero if $E'$ is not an
expression of~$w_i(\sigma_i)$, the formula~\eqref{eq:nshaped} follows.  Notice
that~\eqref{eq:nshaped} does allow us to compute~$X$ when Step~1 and Step~2 are completed.
Consequently, we can compute
the number of shaped $j$-partitions of~$T$ with characteristic~$E$,
which is
\[
\U(T,w,E)-X.
\]
This accomplishes Procedure~\ref{proc:1}.

It remains to design Step~2. To this end, we fix
a~$j$-situation~$\sigma=((\sigma_1,w_1),\dotsc,(\sigma_{t},w_{t}))$.
Define~$\Lambda$ to be the set of all sequences~$(T_1,\dotsc,T_t)$ such that
for each~$i\in\{1,\dotsc,t\}$,
\begin{itemize}
    \item $T_i$ is either a shape of~$T$ or a leaf;
    \item $T_i$ is isomorphic to~$(\sigma_i,w_i)$ as a weighted non-rooted tree; and
    \item if $k\in\{1,\dotsc,t\}\setminus\{i\}$, then $T_i$ is not a subtree of~$T_k$.
\end{itemize}

\begin{observation}\label{obs:Lambda}
   The number of times that $\sigma$ occurs in~$T$ is equal to
   $\abs{\Lambda}$.
\end{observation}
\begin{proof}
We prove that the elements of~$\Lambda$ are exactly occurrences of~$\sigma$ in~$T$.
By the definition, each occurrence of~$\sigma$ gives rise to an element
of~$\Lambda$.

Conversely, let~$(T_1,\dotsc,T_t)$ be an element of~$\Lambda$.
Observation~\ref{obs:step2} implies that the trees~$T_i$ are mutually
disjoint. For each~$i\in\{1,\dotsc,t\}$, let~$e_i$ be the edge of~$T$
associated to the shape~$T_i$, that is, $e_i$ connects $T_i$ to~$T-T_i$;
and let~$v_i$ be the endvertex of~$e_i$ that does
not belong to~$T_i$. Note that $v_i\notin\cup_{k=1}^{t}T_k$ since no tree~$T_i$
is a subtree of another tree~$T_k$ and $j\le \alpha_l < w(T)$.  Set~$T_0'\coloneqq
T$ and~$T_i'\coloneqq T_{i-1}'-T_i$ for~$i\ge1$.

Observe that each of~$T_{i+1},\dotsc,T_t$ is either a leaf or a shape of~$T_i'$.
Hence $T_i'$ is connected and contains all the vertices~$v_1,\dotsc,v_t$.  Therefore
setting~$T'\coloneqq T_{t}'$ shows that $(T_1,\dotsc,T_t)$ occurs in~$T$.
\end{proof}

Our goal is to compute~$\abs{\Lambda}$. For a weighted tree~$(T',w')$,
define~$\Lambda_0(T',w')$ to be the set of all sequences~$(T_1,\dotsc,T_t)$ such
that for each~$i\in\{1,\dotsc,t\}$ it holds that
$T_i$ is either a leaf or a shape of~$T'$ that is isomorphic to~$(\sigma_i,w_i)$ as a weighted non-rooted tree.
Set~$\Lambda_0\coloneqq\Lambda_0(T,w)$. In this
notation, the weight shall be omitted when there is no risk of confusion.  The
advantage of~$\Lambda_0$ is that its size can be computed. Indeed,
\[
\abs{\Lambda_0}=\prod_{i=1}^t \s{(\sigma_i,w_i)}{(T,w)},
\]
where $\s{(\sigma_i,w_i)}{(T,w)}$ is the number of leaves or shapes of~$T$ that are
isomorphic to~$(\sigma_i,w_i)$ as weighted non-rooted trees. This number is given in the input of
Procedure~\ref{proc:1}, since $w_i(\sigma_i)< j$.

Next, we compute~$\abs{\Lambda}$ using the principle of inclusion and
exclusion.  Setting~$I\coloneqq{\{1,\dotsc,t\}}^2\setminus\sst{(i,i)}{1\le i\le
t}$, we have
\[
   \abs{\Lambda}=\abs{\Lambda_0}-\abs{\bigcup_{(i,k)\in
   I}\Lambda_{(i,k)}},
\]
where $\Lambda_{(i,k)}$ is the subset of~$\Lambda_0$ composed of the elements~$(T_1,\dotsc,T_t)$ with $T_i\subseteq T_k$.

By the principle of inclusion-exclusion, we deduce that the output of Step~$2$
is equal to
\[
\abs{\Lambda_0}-\sum_{\varnothing\neq J\subseteq I}
{(-1)}^{\abs{J}-1}\abs{\bigcap_{(i,k)\in J}\Lambda_{(i,k)}}.
\]

It remains to compute~$\abs{\bigcap_{(i,k)\in J}\Lambda_{(i,k)}}$ for each
non-empty subset~$J$ of~$I$.  We start with an observation, which
characterises the sets~$J$ for which the considered intersection is not empty.

\begin{observation}\label{obs:J}
    Let $J\subseteq I$. Then, $\bigcap_{(i,k)\in
    J}\Lambda_{(i,k)}\neq\varnothing$ if and only if for every $(i,k)\in J$,
    either $\sigma_i$ is isomorphic to~$\sigma_k$, or $\sigma_k$ has a leaf or a shape
    that is isomorphic to~$\sigma_i$ as a weighted non-rooted tree.
\end{observation}

From now on, we consider an arbitrary contributing set~$J$.  We construct four
directed graphs~$A_0$, $A_1$, $A_2$ and~$A_3$ that depend on~$J$.  Each vertex~$x$
of~$A_l$ is labeled by a subset~$\ell(x)$
of~$\{(\sigma_1,w_1),\dotsc,(\sigma_t,w_t)\}$. These labels will have the
following properties.
\begin{enumerate}
    \item ${(\ell(x))}_{x\in V(A_l)}$ is a partition of~$\{(\sigma_1,w_1),\dotsc,(\sigma_t,w_t)\}$.
    \item For each vertex~$x$ of~$A_l$, all weighted trees in~$\ell(x)$ are
       isomorphic.
    \item $\abs{\cap_{(i,k)\in J}\Lambda_{(i,k)}}$ is equal to the number of
       elements~$(T_1,\dotsc,T_t)$ of~$\Lambda_0$ such that
    \begin{itemize}
        \item for each vertex~$x$ of~$A_l$, if
           $(\sigma_i,w_i),(\sigma_k,w_k)\in\ell(x)$ then $T_i=T_k$; and
        \item for every arc~$(x,y)$ of~$A_l$, if
           $((\sigma_i,w_i),(\sigma_k,w_k))\in\ell(x)\times\ell(y)$, then
           $T_i\subseteq T_k$.
    \end{itemize}
\end{enumerate}

The directed graph~$A_0$ is obtained as follows. We start from the vertex
set~$\{z_1,\dotsc,z_t\}$. For each~$i\in\{1,\dotsc,t\}$, the label~$\ell(z_i)$
of~$z_i$ is set to be~$\{(\sigma_i,w_i)\}$.  For each~$(i,k)\in J$, we add an arc
from~$z_i$ to~$z_k$. Thus $A_0$ satisfies properties~$(1)$--$(3)$.  Note that
$A_0$ may contain directed cycles, but by Observation~\ref{obs:J}, if $C$ is
a directed cycle then all elements in~$\cup_{x\in V(C)}\ell(x)$ are isomorphic.

Now, $A_1$ is obtained from~$A_0$ by the following recursive operation.  Let~$(x,y,z)$ be
a triple of vertices such that $(x,y)$ and~$(x,z)$ are arcs, but neither~$(y,z)$
nor~$(z,y)$ are arcs.  Let~$(\sigma_y,w_y)\in\ell(y)$
and~$(\sigma_z,w_z)\in\ell(z)$. We add the arc~$(y,z)$ if
$\abs{V(\sigma_y)}\le\abs{V(\sigma_z)}$, and the arc~$(z,y)$ if
$\abs{V(\sigma_z)}\le \abs{V(\sigma_y)}$. (In particular, if
$\abs{V(\sigma_y)}=\abs{V(\sigma_z)}$, then both arcs are added.)

We observe that $A_1$ satisfies~$(1)$--$(3)$. Since neither the vertices nor the
labels were changed, the only thing that we need to show is that if the
arc~$(y,z)$ was added, then for all sequences $(T_1,\dotsc,T_t)\in\cap_{(i,j)\in
J}\Lambda_{(i,j)}$ and
all~$((\sigma_i,w_i),(\sigma_k,w_k))\in\ell(y)\times\ell(z)$, it holds that
$T_i\subseteq T_k$.  This follows from Observation~\ref{obs:step2}: since $(y,z)$
was added, there exists~$s\in\{1,\dotsc,t\}$ such that $T_s$ is contained in
both~$T_i$ and~$T_k$.

The directed graph~$A_2$ is obtained from~$A_1$ by recursively contracting all
directed cycles of~$A_1$. Specifically, for each directed cycle~$C$, all the
vertices of~$C$ are contracted into a vertex~$z_C$ (parallel arcs are removed,
but not directed cycles of length~$2$), and $\ell(z_C)\coloneqq\cup_{x\in
V(C)}\ell(x)$.  We again observe that $A_2$ satisfies properties~$(1)$--$(3)$.

Finally, $A_3$ is obtained from~$A_2$ by recursively deleting transitivity
arcs, that is, the arc~$(y,z)$ is removed if there exists a directed path of
length greater than~$1$ from~$y$ to~$z$.  Note that~$A_2$ and~$A_3$ have the
same vertex-set, and every arc of~$A_3$ is also an arc in~$A_2$.  Again, $A_3$
readily satisfies properties~$(1)$--$(3)$.

Now, let us prove that each component of~$A_3$ is an arborescence, that is
a directed acyclic graph with each out-degree at most one.  We only need to
show that every vertex of~$A_3$ has outdegree at most~$1$.  Assume that $(x,y)$
and~$(x,z)$ are two arcs of~$A_3$. First, note that, in~$A_2$, there is no
directed path from~$y$ to~$z$ or from~$z$ to~$y$, for otherwise the arc~$(x,y)$
or the arc~$(x,z)$ would not belong to~$A_3$, respectively. Therefore,
regardless of whether~$y$ and~$z$ arose from contractions of directed cycles
in~$A_1$, there exist three vertices~$x'$, $y'$ and~$z'$ in~$A_1$ such that
both~$(x',y')$ and~$(x',z')$ are arcs but neither~$(y',z')$ nor~$(z',y')$ is an
arc.  This contradicts the definition of~$A_1$.  Consequently, every vertex
of~$A_3$ has outdegree at most~$1$, as wanted.

We define~$\tau_i$ to be the sequence
\[
      (\s{(\sigma_{i},w_i)}{(T,w)},\s{(\sigma_{i},w_i)}{(\sigma_1,w_1)},\dotsc,\s{(\sigma_{i},w_i)}{(\sigma_t,w_t)})
\]
We recall that $\tau_1,\dotsc,\tau_t$ are known from the assumptions
of Procedure~\ref{proc:1}. Step~$2$ is completed by the following procedure.
\begin{procedure}\label{proc:3}\mbox{}\\
\textsc{input:} A labeled directed forest~$A$ of arborescences and the
sequences $\tau_1,\dotsc,\tau_t$.\\
\textsc{output:} For each~$H\in\{(T,w),(\sigma_1,w_1),\dotsc,(\sigma_t,w_t)\}$,
the number~$\mathcal{P}_3(H,A,\tau(T))$
of elements $(T_1,\dotsc,T_t)$ of~$\Lambda_0(H)$ such that
        \begin{itemize}
        \item for each vertex~$x$ of~$A$, if
        $(\sigma_i,w_i),(\sigma_k,w_k)\in\ell(x)$
        then $T_i=T_k$; and
        \item for every arc~$(x,y)$ of~$A$, if
        $((\sigma_i,w_i),(\sigma_k,w_k))\in\ell(x)\times\ell(y)$, then
        $T_i\subseteq T_k$.
        \end{itemize}
\end{procedure}

The output of Procedure~\ref{proc:3} can be recursively computed as follows.
Let~$V_{\max}$ be the set of vertices of~$A$ with outdegree~$0$.  For each
vertex~$x$ of~$A$, let~$(\sigma^x,w^x)$ be a representative of~$\ell(x)$.
\[
\mathcal{P}_3(H,A,\tau(T))=\negthickspace\negthickspace\prod_{x\in
V_{\max}}\negthickspace\negthickspace\left(\s{(\sigma^x,w^x)}{H}\right)\cdot
\mathcal{P}_3((\sigma^x,w^x),\tilde{A}(w),\tau(T)),
\]
where $\tilde{A}(w)$ is obtained from the component of~$A$ that contains~$x$ by removing~$x$.

By property~(3) of the labels, the output~$\mathcal{P}_3(T,A_3,\tau(T))$ is equal
to~$\abs{\cap_{(i,k)\in J}\Lambda_{(i,k)}}$. This concludes the design of
Procedure~\ref{proc:1}.

\medskip
\noindent
\textbf{Acknowledgements.} The authors thank the referees for their numerous
suggestions on a preliminary version of this work, including pointing out some inaccuracies.


\end{document}